\documentclass[sn-mathphys]{sn-jnl}

\usepackage{anyfontsize}
\usepackage{amssymb}
\usepackage{cmap}
\usepackage{hyperxmp}
\usepackage{mathtools}
\usepackage{tikz-cd}
\usepackage{program}
\mathtoolsset{showonlyrefs}
\usepackage{enumerate}
\usepackage{todonotes}

\usepackage{lineno}

\jyear{2021}%

\theoremstyle{thmstyleone}%
\newtheorem{thm}{Theorem}[section]
\newtheorem{cor}[thm]{Corollary}
\newtheorem{prop}[thm]{Proposition}%
\newtheorem{lem}[thm]{Lemma}

\theoremstyle{thmstyletwo}%
\newtheorem{rem}{Remark}%

\theoremstyle{thmstylethree}%

\newcommand{\id}{\mathrm{id}}           

\newcommand{\norm}[1]{\left\Vert#1\right\Vert}          

\let\on=\operatorname

\def\img{\on{im}}

\begin{document}

\title[]{A vector field induced de Rham-Hodge theory on manifolds}


\author[1]{\fnm{Zhe} \sur{Su}}
\email{zhs0011@auburn.edu}

\affil[1]{\orgdiv{Department of Mathematics and Statistics}, \orgname{Auburn University}, \orgaddress{\street{221 Roosevelt Concourse}, \city{Auburn}, \postcode{36849}, \state{CA}, \country{USA}}}


\abstract{
    We introduce a de Rham-Hodge framework induced by a vector field on a compact, oriented smooth manifold. Using a vector field induced bundle isomorphism on differential forms, we define a vector field induced Hodge $L^2$-inner product, codifferential, and Hodge Laplacian. Unlike classical deformations, such as the drifting and Witten-type Hodge Laplacians, the induced Laplacian modifies the principal symbol and gives rise to an anisotropic Laplace-Beltrami type operator on functions.
    We establish the resulting de Rham-Hodge theory for closed manifolds, including the ellipticity of the induced Hodge Laplacian and the corresponding Hodge decomposition and isomorphism results. We further extend the framework to manifolds with boundary by imposing certain vector field induced boundary conditions, which are necessary to restore the adjointness between the differential and induced codifferential, and to obtain a well-posed boundary value problem. Under these boundary conditions, we establish analogues of the Hodge–Morrey and Friedrichs decompositions. We also discuss several structural properties of the framework, including its relation to anisotropic Laplace–Beltrami operators, its spectral behavior in several explicit examples, and its invariance under isometries.
}

 \keywords{de Rham-Hodge theory; Hodge Laplacian; manifolds with boundary; anisotropic Laplace–Beltrami operators}

 \pacs[MSC Classification]{58A14, 58J32}

\maketitle



    




\section{Introduction}

The de Rham-Hodge theory is built on the de Rham complex of differential forms, where the nilpotency $dd=0$ leads to the de Rham cohomology. For closed manifolds, the theory connects the de Rham cohomology classes of a manifold with the harmonic forms, which is established through the Hodge decomposition following from the ellipticity of the Hodge Laplacian \cite{hodge1989theory}. For manifolds with boundary, this identification also holds by imposing appropriate boundary conditions, such as the normal (Dirichlet) and tangential (Neumann) boundary conditions \cite{shonkwiler2009poincare}. Under these boundary conditions, the Hodge Laplacian then gives rise to elliptic boundary value problems, which lead to the Hodge-Morrey decomposition \cite{morrey1956variational}, Friedrichs decomposition \cite{friedrichs1955differential}, and related results. In particular, one has the identifications of the normal harmonic forms with the relative de Rham cohomology and of the tangential harmonic forms with the absolute de Rham cohomology, which captures the topology of the underlying manifold and its boundary. For a detailed treatment of the theory, we refer to Warner's book~\cite{warner1983foundations} for closed manifolds and Schwartz's book~\cite{schwarz2006hodge} for manifolds with boundary.


Beyond the underlying manifold itself, the de Rham-Hodge theory has also been extended to manifolds equipped with additional structures. For weighted manifolds, one can incorporate a function through a weighted measure on the manifold and a modified codifferential, which gives rise to the drifting Hodge Laplacian and the related Hodge decomposition and Hodge isomorphism results. The corresponding theory was established by Bueler in \cite{bueler1999heat} for non-compact manifolds, with results remaining valid for compact manifolds. The drifting Laplacian on functions goes back to Bakry-Emery's earlier work \cite{bakry2006diffusions} in the study of the diffusion processes on manifolds. The associated operators have also found applications in research on Ricci flow \cite{perelman2002entropy,lott2003some} and spectral analysis \cite{branding2024eigenvalue}. 

Another important extension leads to the Witten-Hodge theory, which was introduced by Witten in the context of supersymmetric quantum mechanics and Morse theory \cite{witten1982supersymmetry} and is closely related to the topological quantum field theory \cite{atiyah1988topological}. In this framework, both the differential and the codifferential are modified to keep the symmetry. The formulation can be given by incorporating either a function or a killing vector field on the manifold. For manifolds with boundary, the corresponding Witten Hodge theory and equivariant cohomology have been studied by Al-Zamil and Montaldi in~\cite{al2012witten}. Closely related to this is Morse-Novikov cohomology \cite{novikov1982hamiltonian}, which can be obtained by replacing the exact one-form in the Witten deformation by a closed one-form. The theory was first studied by Lichnerowicz~\cite{lichnerowicz1977varietes} in the context of Poisson geometry, and has been used to study the local conformal symplectic manifolds~\cite{haller1999group,de2003computation}.

Recent developments in topological data analysis have further stimulated interest in extensions of de Rham-Hodge theory, including the persistent Hodge Laplacians~\cite{chen2021evolutionary, su2025topological}. These developments highlight the usefulness of adapting Hodge-theoretic tools to incorporate additional geometric structures. In many applications, the underlying data may exhibit directional information, which is not directly captured by the standard de Rham-Hodge framework. Existing deformations, such as Witten and Morse-Novikov theories, however, are primarily based on vector fields (or one-forms) under additional constraints, requiring them to be killing or closed. 

In this paper, we introduce a de Rham-Hodge framework induced by a general smooth vector field $v$ on the manifold, without imposing additional constraints. More precisely, we introduce a $v$-induced bundle isomorphism
\begin{align}
    T_v = \id + v^\flat\wedge\iota_v,
\end{align}
which preserves the degree of differential forms. This operator induces a modified inner product on differential forms, and gives rise to the corresponding Hodge star operator, codifferential and Hodge Laplacian. In contrast to the drifing and Witten-type Hodge Laplacians, which differ from the standard Hodge Laplacian only by lower-order terms, the $v$-induced Hodge Laplacian modifies the principal symbol and gives rise to anisotropic diffusion operators. In particular, the resulting Laplacian on functions belongs to the class of anisotropic Laplace-Beltrami type operators, which have appeared in the study of anisotropic geometric processing and diffusion problems \cite{clarenz2000anisotropic, bajaj2003anisotropic}.
 
 We develop the resulting de Rham-Hodge theory using this construction. On closed manifolds, we prove the ellipticity of the induced Hodge Laplacian, provide the corresponding Hodge decomposition, and identify the space of the induced harmonic forms with the de Rham cohomology. For manifolds with boundary, the presence of the vector field requires modified boundary conditions, which are necessary to restore the adjointness between the differential and the modified codifferential, and to obtain well-posed elliptic boundary value problems. Under these boundary conditions, we establish the analogues of Hodge-Morrey and Friedrichs decompositions, as well as the corresponding Hodge isomorphism theorems. We also discuss several structural properties of the v-induced framework, including its relation to the Laplace-Beltrami operators and several explicit examples to illustrate the spectral effects of the vector field, a condition under which a harmonic vector field remains $v$-harmonic in its own induced framework, and its invariance under isometries.

\section{$v$-induced $L^2$-inner product and differential operators}

In this section, we first recall the differential operators and the Hodge $L^2$-inner product in the standard de Rham-Hodge framework. We then introduce a vector field induced isomorphism on differential $k$-forms, and use it to construct the corresponding vector field induced differential operators and inner product. We will also discuss several properties of these operators that will be needed later to build the vector field induced de Rham-Hodge framework.

Let $M$ be a compact, oriented smooth manifold (with or without boundary). Denote by $\Omega^k(M)=\Gamma(\wedge^kT^*M)$ the space of differential $k$-forms on $M$. Each $k$-form on $M$ is a smooth section of the $k$-th exterior power of the cotangent bundle of $M$, namely a smooth antisymmetric covariant tensor field on $M$ of degree $k$. The differential $d$, also called the exterior derivative, maps $k$-forms to $(k\!+\!1)$-forms. It satisfies the Leibniz rule with respect to the wedge product and has the property $dd=0$. A $k$-form $\omega$ is called closed if $d\omega = 0$, and exact if there is another form $\eta\in\Omega^{k-1}(M)$ such that $\omega = d\eta$.

Let $g$ be a Riemannian metric on $M$. It then induces a metric $g^{-1}$ on the cotangent bundle $T^*M$, and thus a pointwise inner product on $\Omega^k(M)$ as follows: given $\omega = \omega_1\wedge\ldots\wedge\omega_k$ and $\eta = \eta_1\wedge\ldots\wedge\eta_k$, we have
\begin{align}\label{eq:pointwise_metric_g}
    \langle\omega, \eta\rangle_g = \det\left( g^{-1}(\omega_i, \eta_j) \right).
\end{align}
The Hodge star $\star$ provides an isomorphism between the spaces of $k$-forms and $(n\!-\!k)$-forms, and it is defined by the following formula
\begin{align}
    \omega\wedge\star\eta = \langle\omega, \eta\rangle_g\mu_g,
\end{align}
where $\mu_g$ is the volume form on $M$ induced by $g$. The Hodge $L^2$-inner product is then given by taking the integral of the formula above
\begin{align}\label{eq:hodgeL2}
    (\omega, \eta) = \int_M\langle\omega, \eta\rangle_g\mu_g = \int_M\omega\wedge\star\eta.
\end{align}
The codifferential $\delta$ is defined as 
\begin{align}
    \delta = (-1)^{k}\star^{-1} d \star,
\end{align}
and it also satisfies $\delta\delta = 0$. A $k$-form $\omega$ is coclosed if $\delta\omega=0$, and coexact if there is $\eta\in\Omega^{k+1}(M)$ such that $\delta\eta = \omega$. 
When $M$ has no boundary, the codifferential $\delta$ is the adjoint operator of $d$ with respect to the inner product \eqref{eq:hodgeL2}, i.e., $(d\omega, \eta) = (\omega, \delta\eta)$. However, this is no longer the case when there is a boundary due to the boundary term. 

The Hodge Laplacian is defined by $\Delta = d\delta + \delta d$. We denote by $H^k_{\Delta}(M)$ the kernel of its restriction on $k$-forms, which is called the space of harmonic $k$-forms. The following space
\begin{align}
    \mathcal{H}^k(M) = \{\omega\in\Omega^k(M)\,\vert\, d\omega = 0, \delta\omega =0\}
\end{align}
is known as the space of harmonic fields \cite{kodaira1949harmonic}, which consists of the closed and coclosed $k$-forms. It agrees with the space of harmonic forms $\mathcal{H}_{\Delta}(M)$ when $M$ has no boundary. In the presence of a boundary, this identification no longer holds. The space $\mathcal{H}^k(M)$ is only a subset of $\mathcal{H}_{\Delta}(M)$, i.e., $\mathcal{H}^k(M)\subset \mathcal{H}_{\Delta}(M)$ \cite{schwarz2006hodge}.

\subsection{$v$-induced isomorphism on $k$-forms}

Let $v\in\mathcal{X}(M)$ be a vector field. It then induces an interior product $\iota_v:\Omega^k\to\Omega^{k-1}$ as follows: given a $k$-form $\omega$, $\iota_v(\omega)$ is obtained by inserting $v$ into the first slot, i.e.,
\begin{align}
    \iota_v(\omega)(v_1, \ldots, v_{k-1}) = \omega(v, v_1,\ldots,v_{k-1}),
\end{align}
where each $v_i,\, i=1,\ldots,k\!-\!1$ is also a vector field on $M$. The interior product operator satisfies $\iota_v\iota_v = 0$. 

Denote by $v^\flat = g(v,\cdot)$ the corresponding $1$-form of $v$ in $\Omega^1(M)$ by lowering its index. Notice that the interior product lowers the degree of a form by one, while the wedge product raises it by one. We consider the following map:
\begin{align}
T_v = \on{id} + v^{\flat}\wedge \iota_v:\,\Omega^k(M)\to\Omega^k(M),
\end{align}
where $\on{id}$ is the identity operator, and $\iota_vf = 0$ for $f\in\Omega^0(M)$ by convension. It is clear that $T_v$ preserves the degrees. In the following, we show that it is linear, positive definite, and self-adjoint with respect to the pointwise inner product~\eqref{eq:pointwise_metric_g}. Moreover, $T_v$ is an isomorphism. 

We first need the following lemma and corollary.

\begin{lem}\label{lem.wedge.interiorprod.adjoint}
    Let $\zeta\in\Omega^{k-1}(M)$ and $\eta\in\Omega^k(M)$. Then
    \begin{align}
        \langle v^\flat\wedge\zeta, \eta\rangle_g = \langle\zeta, \iota_v(\eta)\rangle_g.
    \end{align}
\end{lem}
\begin{proof}
We define a $k$-form $\bar{\zeta} = \zeta_1\wedge\zeta_2\wedge\ldots\wedge\zeta_k$, where $\zeta_1 = v^\flat$ and $\zeta = \zeta_2\wedge\ldots\wedge\zeta_k$. By direct computation, we obtain
\begin{align}
        \langle v^\flat\wedge\zeta, \eta\rangle_g = \langle \bar\zeta, \eta\rangle_g &= \det\left(g^{-1}(\zeta_i, \eta_j) \right).
    \end{align}
Note that the entries in the first row of this determinant are given by $g^{-1}(\zeta_1, \eta_j) = g^{-1}(v^\flat, \eta_j) = \eta_j(v)$. We expand it along the first row and obtain
\begin{align}
    \det\left(g^{-1}(\zeta_i, \eta_j) \right) &=\sum_{j=1}^k(-1)^{j-1}\eta_j(v)\langle\zeta_2\wedge\ldots\wedge\zeta_k, \eta_1\wedge\ldots\wedge\hat{\eta_j}\wedge\ldots\wedge\eta_k\rangle_g\\
    &=\sum_{j=1}^k(-1)^{j-1}\eta_j(v)\langle\zeta, \eta_1\wedge\ldots\wedge\hat{\eta_j}\wedge\ldots\wedge\eta_k\rangle_g\\
    &=\langle\zeta, \sum_{j=1}^k(-1)^{j-1}\eta_j(v)\eta_1\wedge\ldots\wedge\hat{\eta_j}\wedge\ldots\wedge\eta_k\rangle_g\\
    &=\langle\zeta, \iota_v(\eta)\rangle_g.
\end{align}
This proves the identity.
\end{proof}
\begin{cor}\label{cor.innermul}
    Let $\omega, \eta\in\Omega^{k}(M)$. Then $\langle v^\flat\wedge \iota_v(\omega), \eta\rangle_g = \langle \iota_v(\omega), \iota_v(\eta)\rangle_g.$
\end{cor}
We immediately have the following proposition.
\begin{prop}
    Let $\omega, \eta\in\Omega^{k}(M)$. Then 
    \begin{align}
        \langle T_v(\omega), \eta\rangle_g = \langle\omega, \eta\rangle_g + \langle \iota_v(\omega), \iota_v(\eta)\rangle_g.
    \end{align}
    Consequently, $T_v$ is positive definite and self-adjoint with respect to the pointwise metric~\eqref{eq:pointwise_metric_g} on $\Omega^k(M)$, i.e.,
    \begin{align}\label{eq:Tv.selfadjoint}
        \langle T_v(\omega), \eta\rangle_g = \langle \omega, T_v(\eta)\rangle_g.
    \end{align}
\end{prop}
\begin{proof}
    By Corollary~\ref{cor.innermul}, we compute
    \begin{align}
        \langle T_v(\omega), \eta\rangle_g &= \langle \omega + v^\flat\wedge \iota_v(\omega), \eta\rangle_g = \langle\omega, \eta\rangle_g + \langle v^\flat\wedge \iota_v(\omega), \eta\rangle_g\\
        &= \langle\omega, \eta\rangle_g + \langle \iota_v(\omega), \iota_v(\eta)\rangle_g.
    \end{align}
    The self-adjointness of $T_v$ follows immediately. Replacing $\eta$ with $\omega$, we obtain
    \begin{align}
        \langle T_v(\omega), \omega\rangle_g = \norm{\omega}_g^2 + \norm{\iota_v(\omega)}\geq 0,
    \end{align}
    which shows the positive definiteness. 
\end{proof}
We next establish a formula for the inverse of $T_v$, and thus $T_v$ is an isomorphism.
\begin{prop}
The inverse of $T_v$ is given by
    \begin{align}
        T_v^{-1} = \on{id} - \frac{1}{1+\norm{v}^2_g}v^\flat\wedge \iota_v.
    \end{align}
    As a result, $T_v:\Omega^k(M)\to\Omega^k(M)$ is an isomorphism.
\end{prop}
\begin{proof}
By direct computation, we have
\begin{align}
    \iota_v(v^\flat\wedge \iota_v) &= \iota_v(v^\flat)\iota_v + (-1)v^\flat\wedge \iota_v\iota_v\\
    &= \norm{v}^2_g\iota_v + (-1)v^\flat\wedge \iota_v\iota_v = \norm{v}^2_g\iota_v.
\end{align}
Thus
    \begin{align}
        \left(T_v\right)^{-1} T_v &= \left(\on{id} - \frac{1}{1+\norm{v}^2_g}v^\flat\wedge \iota_v\right)\left(\on{id} + v^\flat\wedge \iota_v\right)\\
        &= \on{id} - \frac{1}{1+\norm{v}^2_g}v^\flat\wedge \iota_v + v^\flat\wedge \iota_v - \frac{1}{1+\norm{v}^2_g}v^\flat\wedge \iota_v(v^\flat\wedge \iota_v)\\
        &= \on{id} - \frac{1}{1+\norm{v}^2_g}v^\flat\wedge \iota_v + v^\flat\wedge \iota_v - \frac{\norm{v}^2_g}{1+\norm{v}^2_g}v^\flat\wedge \iota_v = \on{id}.
    \end{align}
One can verify that $T_v\left(T_v\right)^{-1}=\on{id}$ in an analogous way. Since $T_v$ is invertible, it is an isomorphism.
\end{proof}

\subsection{$v$-induced $L^2$-inner product and operators}\label{sec:innerProd_and_operators}

Using the isomorphism $T_v$, we now introduce the corresponding $v$-induced Hodge structure on differential forms. The operator $T_v$ induces a pointwise inner product on $k$-forms by
\begin{align}\label{eq:pointwise_metric_g.vinduced}
    \langle\omega, \eta\rangle_{g, v} = \langle T_v(\omega), \eta\rangle_{g}.
\end{align}
Integrating over $M$, we obtain the corresponding $v$-induced Hodge $L^2$-inner product, given by
\begin{align}\label{eq:HodgeL2.v.induced}
    (\omega, \eta)_v = \int_M\langle\omega, \eta\rangle_{g, v}\mu_g = \int_M\langle T_v(\omega), \eta\rangle_gd\mu_g.
\end{align}
We denote the $v$-induced norm as $\norm{\cdot}_v$. This inner product can also be reformulated by introducing a $v$-induced Hodge star operator, given by $\star_v = \star T_v:\Omega^k(M)\to\Omega^k(M)$, which is an isomorphism since the Hodge star $\star$ and $T_v$ are both isomorphisms. By the self-adjointness of $T_v$, we have
\begin{align}
    (\omega, \eta)_v = \int_M \omega\wedge\star T_v\eta = \int_M \omega\wedge\star_v\eta.
\end{align} 
The $v$-induced Hodge star naturally leads to a $v$-induced codifferential by replacing the standard Hodge star operator in the definition of the codifferential $\delta$. More precisely, we define
\begin{align}\label{eq:def.delta_v}
    \delta_v^k := (-1)^{k}\star_v^{-1}\,d\,\star_v.
\end{align}
By direct computation, we can check that $\delta_v = T_v^{-1}\delta T_v$. In particular, $\delta_v\delta_v = 0$.
We call a $k$-form $\omega$ $v$-coclosed if $\delta_v(\omega) = 0$ and $v$-coexact if there is $\eta\in\Omega^{k+1}(M)$ such that $\omega = \delta_v(\eta)$. When $v = 0$, $T_v = \on{id}$, and then the $v$-induced codifferential $\delta_v$ reduces to the usual codifferential $\delta$. Note that in the standard case, one can also define the codifferential using the formula $\delta = (-1)^{n(n-k)+1}\star d\star$. However, in our $v$-induced framework, the formula~\eqref{eq:def.delta_v} is not equivalent to $(-1)^{n(n-k)+1}\star_v d\star_v$, since $\star_v^{-1}$ and $\star_v$ do not differ merely by a power of the sign, see Lemma~\ref{lem:starTv}.

The differential $d$ and the $v$-induced codifferential $\delta_v$ naturally define a vector field induced Hodge Laplacian, which we call the $v$-induced Hodge Laplacian, given by
\begin{align}\label{eq:laplacian.vinduced}
    \Delta_v := d\delta_v + \delta_vd,
\end{align}
which preserves the degree of differential forms. We denote by
\begin{align}
    \mathcal{H}_{\Delta, v}^k(M) = \{\omega\in\Omega^k(M)\,\vert\, \Delta_v\omega = 0\}
\end{align}
its kernel on $k$-forms, and call elements in this kernel the $v$-harmonic $k$-forms.
In addition, we introduce the space
\begin{align}
    \mathcal{H}^k_v(M) = \{\omega\in\Omega^k(M)\,\vert\, d\omega = 0, \delta_v\omega =0\}
\end{align}
consisting of the closed and $v$-coclosed $k$-forms. We call this space $\mathcal{H}^k_v(M)$ the space of $v$-harmonic fields.

In analogy to the standard theory, a first-order differential operator $D_v = d + \delta_v$, called the $v$-induced Dirac operator, can also be defined. Its square is the $v$-induced Hodge Laplacian $D_v^2 = \Delta_v$. The spectral properties of $\Delta_v$ can also be studied through the corresponding eigenvalues and eigenforms of $D_v$.

\begin{rem}
    One could modify both the differential and codifferential operators in analogy to the Witten-Hodge framework~\cite{witten1982supersymmetry} by defining $d_v = T_vdT_v^{-1}$, which gives rise to a different class of the $v$-induced Hodge Laplacian. However, in this work we focus on deformations induced solely through the codifferential, which produce anisotropic effects at the level of the principal symbol.
\end{rem}

When acting on $0$-forms (i.e., functions), the $v$-induced Hodge Laplacian $\Delta^0_v = \delta_vd = \delta T_v d: C^{\infty}(M)\to C^{\infty}(M)$ belongs to the family of the anisotropic Laplace-Beltrami operators~\cite{clarenz2000anisotropic, bajaj2003anisotropic}. In fact,  in local coordinates, notice that $\delta(\omega_idx^i) = -\frac{1}{\sqrt{G}}\partial_i\left(\sqrt{G}g^{ij}\omega_j\right)$ with $G = \det(g)$. By computation, one immediately obtains the coordinate expression of the Laplacian as follows:
\begin{align}
    \Delta_v^0 \varphi &= -\frac{1}{\sqrt{G}} \partial_i\left(\sqrt{G}T_v^{ij} \partial_j \varphi\right),\quad \varphi\in C^{\infty}(M),
\end{align}
where $(T_v)^{ij}=g^{is}(T_v)^j_s = g^{is}\left(\delta_s^j + v^j v_s\right) = g^{ij} + v^iv^j$ is positive definite symmetric. This direction-dependent diffusion tensor $(T_v)^{ij}$ encodes the anisotropy in the principal symbol of the Laplacian. In the case that $v=0$, $(T_v)^{ij} = g^{ij}$ and then this local formula of $\Delta_v$ on functions reduces to that of the standard Laplace-Beltrami operator on manifolds.



\section{$v$-induced de Rham-Hodge theory for closed manifolds}

In this section, we study the vector field-induced de Rham-Hodge framework on closed manifolds, where $\partial M = \emptyset$. We first establish the adjointness between the differential $d$ and the $v$-induced codifferential $\delta_v$ with respect to the $v$-induced Hodge inner product~\eqref{eq:HodgeL2.v.induced}. We then focus on the basic properties of the $v$-induced Hodge Laplacian, which lead to the $v$-induced Hodge decomposition, and the corresponding Hodge isomorphism that identifies the harmonic forms with the de Rham cohomology.

By the self-adjointness of $T_v$ with respect to the pointwise inner product~\eqref{eq:pointwise_metric_g} and the adjointness between $d$ and $\delta$ with respect to the standard Hodge $L^2$-inner product~\eqref{eq:HodgeL2.v.induced}, one can compute
\begin{align}
    (d\zeta, \eta)_v &= \int_M d\zeta\wedge\star T_v \eta = \int_M\left\langle d\zeta, T_v(\eta)\right\rangle_gd\mu_g\\
    &= (d\zeta, T_v(\eta)) = (\zeta, \delta T_v(\eta)) = \int_M\left\langle \zeta, \delta T_v(\eta)\right\rangle_gd\mu_g\\
    &= \int_M\left\langle T_v^{k-1}(\zeta),\left(T_v^{k-1}\right)^{-1}\delta T_v(\eta)\right\rangle_gd\mu_g\\
    &= \int_M\left\langle T_v^{k-1}(\zeta), \delta_v(\eta)\right\rangle_gd\mu_g =(\zeta, \delta_v\eta)_v.
\end{align}
Therefore, the differential $d$ and the $v$-induced codifferential $\delta_v$ are adjoint with respect to the $v$-induced Hodge inner product~\eqref{eq:HodgeL2.v.induced}.
    


\subsection{Ellipticity of the $v$-induced Hodge Laplacian}

The $v$-induced Hodge Laplacian $\Delta_v$ shares the same basic spectral properties as the standard Hodge Laplacian. Its spectrum consists of nondecreasing discrete real nonnegative eigenvalues with finite multiplicity, and the eigenvectors of different eigenvalues are orthogonal with respect to the $v$-induced $L^2$-inner product. Since $M$ is closed, these results follow directly from the self-adjointness, ellipticity, and positive-semidefiniteness of the $v$-induced Hodge Laplacian.

Using the adjointness between $d$ and $\delta_v$, the following identity can be obtained directly, which implies the self-adjointness of $\Delta_v$:
\begin{align}
(\Delta_v\omega ,\eta)_v = \left(\delta_v\omega, \delta_v\eta\right)_v +  \left(d\omega, d\eta\right)_v = (\omega, \Delta_v\eta)_v.
\end{align}
By taking $\eta=\omega$, the positive semidefiniteness follows from
\begin{align}
(\Delta_v\omega ,\omega)_v = \norm{\delta_v^k\omega}_v^2 + \norm{d^k\omega}_v^2.
\end{align}
In particular, one has $\Delta_v(\omega) = 0$ iff $d\omega = 0$ and $\delta_v(\omega) = 0$. It follows directly that in the case of closed manifolds, the space of harmonic forms agrees with the space of harmonic fields, i.e., $\mathcal{H}^k_{\Delta, v}(M) = \mathcal{H}^k_v(M).$

Next, we show the ellipticity of $\Delta_v$ using its principal symbol directly.
By direct computation, one can see that the difference between the $v$-induced Hodge Laplacian $\Delta_v$ and the standard Hodge Laplacian $\Delta$ contains second-order terms, see for instance the $0$-th order case~\eqref{eq:0thLaplacian2}. Therefore, they do not share the same principal symbol, which is in contrast to the drifting Hodge Laplacian \cite{bueler1999heat} and the Witten Hodge Laplacian \cite{witten1982supersymmetry} that differ from the standard Hodge Laplacian only by lower-order terms. The ellipticity of $\Delta_v$ does not follow immediately from a comparison with the standard Hodge Laplacian $\Delta$. 
\begin{thm}
    The $v$-induced Hodge Laplacian $\Delta_v$ is elliptic.
\end{thm}
\begin{proof}
To establish the ellipticity of $\Delta_v$, we show that for any $x\in M$ and $\xi \neq0$ in $T^*M$, its principal symbol
\begin{align}
    \sigma_{\Delta_v}(\xi): \wedge^kT_x^*M\to \wedge^kT_x^*M 
\end{align}
is an isomorphism. Since that $\wedge^kT_x^*M$ is finite dimensional, it suffices to verify that $\sigma_{\Delta_v}(\xi)$ is injective.

The differential $d$ is a first-order operator. The principal symbol of $\delta_v$, following from the adjointness between $\delta_v$ and $d$ with respect to \eqref{eq:HodgeL2.v.induced}, is then given by the adjoint of $\sigma_d(\xi)$ with a sign \cite{demailly1996theorie}, i.e.,
$$\sigma_{\delta_v}(\xi) = - \sigma_d(\xi)^*.$$
Thus, the principal symbol of $\Delta_v$ is of the form
\begin{align}
    \sigma_{\Delta_v}(\xi) =\sigma_{d}(\xi)\sigma_d(\xi)^* + \sigma_d(\xi)^*\sigma_{d}(\xi).
\end{align}
By computation, we have
\begin{align}
 \left\langle\left(\sigma_{\Delta_v}(\xi)\right)\omega, \omega\right\rangle_{g, v}  = \langle\sigma_{d}(\xi)\omega, \sigma_{d}(\xi)\omega\rangle_{g, v} + \langle\sigma_d(\xi)^*\omega, \sigma_d(\xi)^*\omega\rangle_{g, v}
\end{align}
Therefore, $\sigma_{\Delta_v}(\xi)\omega = 0$ iff $\sigma_d(\xi)\omega = 0$, and $\sigma_d(\xi)^*\omega = 0$.

Note that the principal symbol of the differential $d$ is given by $\sigma_d(\xi) = \xi\wedge$. The wedge multiplication by $\xi$ defines an exact sequence
\begin{align}
    \wedge^{k-1}T^*_xM\xrightarrow{\sigma_d(\xi)\,} \wedge^{k}T^*_xM\xrightarrow{\,\sigma_d(\xi)\,}\wedge^{k+1}T^*_xM.
\end{align} 
From $\sigma_d(\xi)\omega = 0$ we know there is $\eta$ such that $\omega = \sigma_d(\xi)\eta$. Therefore, 
\begin{align}
    \langle\omega, \omega\rangle_{g, v} = \langle \sigma_d(\xi)\eta, \omega\rangle_{g, v} = \langle\eta, \sigma_d(\xi)^*\omega\rangle_{g, v}= 0
\end{align}
It follows that $\omega = 0$ and thus $\sigma_{\Delta_v}(\xi)$ is injective. This proves the statement.
\end{proof}



\subsection{$v$-Hodge decomposition and de Rham cohomology}

In the following, we give the $v$-induced Hodge decomposition on closed manifolds, and connect the space of $v$-harmonic forms to the de Rham cohomology group.

\begin{thm}\label{thm.hodgedecomposition.closed}
The $v$-harmonic space $\mathcal{H}_{\Delta, v}^k(M)$ is finite dimensional. 
One has the orthogonal decomposition
\begin{align}
    \Omega^k(M) = \Delta_v(\Omega^k(M))\oplus \mathcal{H}_{\Delta,v}^k(M)
\end{align}
with respect to the inner product~\eqref{eq:HodgeL2.v.induced}. Moreover,
   \begin{align}
        \Omega^k(M) = d\Omega^{k-1}(M)\oplus\delta_v\Omega^{k+1}(M)\oplus \mathcal{H}_{\Delta,v}^k(M).
    \end{align}
\end{thm}

\begin{proof}
The result follows from the ellipticity and self-adjointness of $\Delta_v$ 
by the standard arguments in Hodge theory for closed manifolds.
\end{proof}


The de Rham complex is a sequence of the space of differential forms linked by the differential $d$, given by
\begin{align*}
0 \rightarrow \Omega^{0}(M)
\xrightarrow{\,d\,}
\Omega^{1}(M)
\xrightarrow{\,d\,}
\cdots
\xrightarrow{\,d\,}
\Omega^{m}(M)
\rightarrow 0,
\end{align*}
and the $k$-th de Rham cohomology group is $H_{dR}^k(M) = \ker d^k/\img d^{k-1}$, where $d^k$ denotes the restriction of $d$ to $k$-forms. This group, by the de Rham theorem, is isomorphic to the $k$-th singular cohomology group, and thus depends only on the topology of the underlying manifold. The dimension of the de Rham cohomology group is given by the $k$-th Betti number $\beta_k$, which measures the number of $k$-dimensional holes in the manifold $M$.

Note that we only modified the codifferential in our framework, with the differential unchanged. The de Rham cohomology is thus independent of the vector field $v$. In particular, we have the following corollary following immediately from the $v$-induced Hodge decomposition.
\begin{cor}
    Each de Rham cohomology class in $H_{dR}^k(M)$ admits a unique $v$-harmonic representative. Thus, one has the isomorphism 
    \begin{align}
        \mathcal{H}_{\Delta, v}^k(M)\cong H_{dR}^k(M).
    \end{align}
\end{cor}

\section{$v$-induced de Rham-Hodge theory for manifolds with boundary}

    


We now consider the case of manifolds with boundary. In the standard de Rham-Hodge theory, the differential $d$ and the codifferential $\delta$, due to the appearance of the boundary, are in general not adjoint with respect to the Hodge inner product~\eqref{eq:hodgeL2} as in the case of closed manifolds. As a consequence, the resulting Hodge theory, including the self-adjointness of the Hodge Laplacian and the Hodge decomposition, fails to hold. The space of harmonic fields $\mathcal{H}^k(M)$ is infinite dimensional and is only a subset of the space of harmonic forms $H_{\Delta}^k(M)$, which does not directly connect to the cohomology \cite{schwarz2006hodge}. To recover these properties, one must impose appropriate boundary conditions under which $d$ and $\delta$ are adjoint. The same is true in our $v$-induced framework.

Let $j:\partial M\to M$ be the inclusion map. In the standard setting with respect to the standard Hodge $L^2$-inner product~\eqref{eq:hodgeL2}, two common choices are the normal (Dirichlet) and tangential (Neumann) boundary conditions \cite{shonkwiler2009poincare}, which require that the pullback of $\omega$ and $\star\omega$ to the boundary $\partial M$ vanish, i.e., $j^*(\omega) = 0$ and $j^*(\star\omega) = 0$, respectively. Equivalently, $\omega$ is normal if $\omega$ vanishes on tangent vectors to the boundary, and it is tangential if the same condition holds for $\star\omega$. We denote by $\Omega^k_n(M)$ the space of normal $k$-forms and by $\Omega^k_t(M)$ the space of tangential forms, given as follows
\begin{align} \label{eq:bc.n}
\Omega^k_n(M) &= \{\omega\in\Omega^k(M)\, \vert\, j^*(\omega) = 0\};\\ \label{eq:bc.t}
\Omega^k_t(M) &= \{\omega\in\Omega^k(M)\, \vert\, j^*(\star\omega) = 0\}.
\end{align}
These two spaces are isomorphic under the Hodge star $\star$.  Once the normal and tangential boundary conditions are imposed, the differential $d$ and codifferential $\delta$ become adjoint with respect to the standard Hodge $L^2$-inner product~\eqref{eq:hodgeL2}. Moreover, $d$ preserves the normal boundary condition, while $\delta$ preserves the tangential boundary condition. In the case that $\partial M = \emptyset$, both spaces reduce to the space of differential forms $\Omega^k(M)$.

\subsection{$v$-induced boundary conditions}

The standard boundary conditions, however, are not sufficient to recover the adjointness between $d$ and $\delta_v$ in our $v$-induced framework, since $T_v$ preserves the boundary conditions only when $v$ is normal or tangential to the boundary, see Appendix~\ref{prop:Tv.bcs} for a proof. Note that we only modify the codifferential with the differential unchanged. To develop a framework that applies to a general vector field, we therefore introduce the following modified tangential boundary condition, which we call the $v$-tangential boundary condition, as follows:
\begin{align}
\label{eq:bc.tv}
\Omega^k_{t, v}(M) &= \{\omega\in\Omega^k(M)\, \vert\, j^*(\star_v\omega) = 0\}.
\end{align}
By definition, the space of $v$-tangential forms $\Omega^k_{t, v}(M)$ is isomorphic to the space of normal forms $\Omega^{m-k}_{n}(M)$ under the $v$-induced Hodge star $\star_v = \star T_v$. When $v$ is normal or tangential to the boundary, the $v$-tangential boundary condition reduces to the standard tangential boundary condition~\eqref{eq:bc.t}, and then the space $\Omega^k_{t, v}(M)$ identifies with the space of tangential forms $\Omega^k_t(M)$. 

\begin{rem}
The geometric meaning of the $v$-tangential boundary condition might be more transparent in the case of $1$-forms. Note that a $1$-form $\omega$ is $v$-tangential iff $T_v\omega = \omega + \omega(v)v^\flat$ satisfies the standard tangential boundary condtion, i.e., $j^*(T_v\omega) = 0$. This is equivalent to $T_v\omega({\bf n})=0$ along $\partial M$, where ${\bf n}$ denotes the outward unit normal vector field to $\partial M$. In other words, $(T_v\omega)^\sharp$ is orthogonal to the normal vector ${\bf n}$ along the boundary. Now if we split $v = v_t + \norm{v_n}_g{\bf n}$ with $v_t$ and $v_n$ being the tangent and normal parts of $v$ to $\partial M$, then 
\begin{align}
T_v\omega({\bf n}) &= (1+\norm{v_n}_g^2)\omega({\bf n}) + \norm{v_n}_g\omega(v_t) \\
&=\omega\left((1+\norm{v_n}_g^2){\bf n} + \norm{v_n}_gv_t\right) = 0.
\end{align}
It follows that the corresponding vector field $\omega^\sharp$ is orthogonal to a $v$-dependent direction along $\partial M$.
\end{rem}

Since $\delta$ preserves the standard tangential boundary condition, it is direct to verify that $\delta_v$ preserves the $v$-tangential boundary condition. In addition, by integration by parts, we have the following Green’s formula
\begin{align}\label{eq:greenformula.vinduced}
    (d\zeta, \eta)_v  =(\zeta, \delta_v\eta)_v + \int_{\partial M}\zeta\wedge\star T_v\eta.
\end{align}
One can see clearly that the differential $d$ and the $v$-induced codifferential $\delta_v$ are adjoint with respect to the $v$-induced $L^2$-inner product~\eqref{eq:HodgeL2.v.induced} when restricted to the normal and $v$-tangential forms.

Note that the Laplacian is a second-order operator. It is, in general, not sufficient to obtain a well-defined elliptic boundary value problem and its self-adjointness when one restricts it to the space of normal forms or tangential forms with just one boundary condition. Therefore, we introduce the following two subspaces, with each equipped with two boundary constraints.
\begin{align} \label{eq:bc.n.sub}
\bar\Omega^k_{n, v}(M) &= \{\omega\in \Omega^k(M)\, \vert\, j^*(\omega)=0\, ,\,  j^*(\delta_v\omega) = 0\};\\ \label{eq:bc.tv.sub}
\bar\Omega^k_{t, v}(M) &= \{\omega\in \Omega^k(M)\, \vert\, j^*(\star_v\omega)  = 0\, ,\,  j^*(\star_v d\omega) = 0\}.
\end{align}
These two subspaces remain $v$-Hodge dual to each other following Lemma~\ref{lem:starstar}, i.e., $\bar\Omega^k_{n, v}(M)\overset{\star_v}{\cong}\bar\Omega^{m-k}_{t, v}(M)$. Moreover, using Green's formula~\eqref{eq:greenformula.vinduced}, we obtain the following energy identity under these boundary constraints
\begin{align}\label{eq:energyidentity}
    (\Delta_v\omega, \eta)_v = (d\omega, d\eta)_v + (\delta_v\omega, \delta_v\eta)_v.
\end{align}
We also introduce two subspaces of the space of harmonic fields $\mathcal{H}^k_v(M)$, called the normal $v$-harmonic fields and the space of tangential $v$-harmonic fields, respectively, given by
\begin{align} \label{eq.bc.n.subspace}
\mathcal{H}^k_{n, v}(M) = \Omega^k_n(M)\cap\mathcal{H}^k_v(M)\quad\text{ and }\quad \mathcal{H}^k_{t,v}(M) = \Omega^k_{t, v}(M)\cap\mathcal{H}^k_v(M),
\end{align}
which are also isomorphic under the $v$-induced Hodge star, i.e., $\mathcal{H}^k_{t, v}(M)\overset{\star_v}{\cong}\mathcal{H}^{m-k}_{n, v}(M)$.
Denote by $\Delta_{n,v}$ and $\Delta_{t, v}$ the restrictions of the $v$-induced Hodge laplacian $\Delta_v$ to the subspaces $\bar\Omega^k_{n, v}(M)$ and $\bar\Omega^k_{t, v}(M)$, respectively. It follows immediately from the energy identity~\eqref{eq:energyidentity} that the kernels of the two operators agree with the two subspaces of harmonic fields. Specifically, we have
\begin{align}
    \ker\Delta_{n,v}^k = \mathcal{H}_{n, v}^k(M)\quad\text{ and }\quad\ker\Delta_{t,v}^k = \mathcal{H}_{t, v}^k(M).
\end{align}
When $v=0$, the spaces $\mathcal{H}^k_{n,v}(M)$ and $\mathcal{H}^k_{t,v}(M)$ reduce to the space of normal harmonic fields $\mathcal{H}^k_n(M) = \Omega^k_n(M)\cap\mathcal{H}^k(M)$ and the space of tangential harmonic fields $\mathcal{H}^k_t(M) = \Omega^k_t(M)\cap\mathcal{H}^k(M)$, which then coincide with the kernels of the standard Hodge Laplcian $\Delta^k$ under the normal and tangential boundary conditions, respectively.

\subsection{Elliptic boundary value problem}


From now on, we pass from smooth forms to the standard Sobolev setting, which provides the natural framework for the elliptic theory of boundary value problems and is needed for the decomposition results in the next subsection. For the standard Hodge Laplacian, the ellipticity of the associated boundary value problem has been established through the Lopatinski\u{\i}-\v{S}apiro condition ~\cite{schwarz2006hodge}. These elliptic results are crucial for the Hodge decompositions on manifolds with boundary, including the Hodge-Morrey decompositions \cite{morrey1956variational}, Friedrichs decomposition \cite{friedrichs1955differential}, and related results, which we will discuss in the next subsection.

Below, we show that the $v$-induced Hodge Laplacian operators with the aforementioned boundary constraints are self-adjoint and define elliptic boundary value problems. These results are the analogues of the corresponding statements of the standard Hodge Laplacian on the manifolds with boundary. 

We consider the $v$-induced Laplacian as an operator on the $H^2$-completion of the space $\bar\Omega^k_{n,v}(M)$, given by
\begin{align}\label{eq:lap.n}
 \Delta_{n,v} &: H^2\bar\Omega^k_{n, v}(M) \to L^2\Omega^k(M),
\end{align}
with $L^2\Omega^k(M)$ being the $L^2$-completion of the space of smooth forms $\Omega^k(M)$. With a slight abuse of notation, here we continue to use the same notation $\Delta_{n,v}$ for the operator in the Sobolev setting as for the restriction of $\Delta_v$ to the space $\bar\Omega^k_{n, v}(M)$ of smooth forms. The associated boundary value problem is:
\begin{equation}\label{eq:bvp}
\begin{aligned}
&\Delta_{n, v} \omega = \eta &&\quad\text{ on } M,\\
j^* \omega = 0 \quad &\text{and} \quad j^*(\delta_v \omega)=0 &&\quad\text{ on } \partial M.
\end{aligned}
\end{equation}
By the same energy identity formula~\eqref{eq:energyidentity} as in the smooth setting, the operator $\Delta_{n,v}$ is self-adjoint and positive-semidefinite with respect to the inner product~\eqref{eq:HodgeL2.v.induced}. We show below that the boundary value problem~\eqref{eq:bvp} for the operator $\Delta_{n,v}$ is elliptic. Since $M$ is compact, these properties of $\Delta_{n,v}$ ensure that it is Fredholm of index $0$ with finite dimensional kernel and cokernel and closed range. In addition, its spectrum is discrete and consists of a nondecreasing sequence of non-negative real eigenvalues with finite multiplicity, and the corresponding eigenspaces are mutually orthogonal. The corresponding results for the operator $\Delta_{t, v}: H^2\bar\Omega^k_{t,v}(M)\to L^2\Omega^k(M)$ under the dual boundary conditions follow from the same arguments. Here we omit the details.

We now state the following theorem.

\begin{thm}
    The boundary value problem~\eqref{eq:bvp} is elliptic in the sense of Lopatinski\u{\i}--\v{S}apiro condition.
\end{thm}
\begin{proof}
Since the ellipticity is a local property, the operator $\Delta_{n,v}$ is elliptic in the interior exactly as in the case of closed manifolds. For the boundary conditions, note that $T_v$ is an invertible bundle map of order $0$, and $\delta_v = T_v^{-1}\delta T_v$. The principal symbol of the boundary operator $j^*(\delta_v\omega)$ differs from the standard boundary operator $j^*(\delta\omega)$ only by composition with $T_v$ and its invertible map $T_v^{-1}$, which does not affect the ellipticity of the boundary condition. As a result, the boundary value problem~\eqref{eq:bvp} of $\Delta_{n,v}$ is elliptic in the sense of Lopatinski\u{\i}--\v{S}apiro condition, just as in the corresponding boundary value problem for the standard Hodge Laplacian $\Delta$.
\end{proof}

The elliptic regularity of $\Delta_{n,v}$ yields immediately the following corollary.
\begin{cor}\label{cor:harmonicfieldissmooth}
    If $\eta\in\Omega^k(M)$ is a smooth $k$-form and $\omega\in H^2\Omega^k(M)$ is a strong solution of~\eqref{eq:bvp}, then $\omega$ is also smooth.
\end{cor}

It follows directly from Corollary~\ref{cor:harmonicfieldissmooth} that all elements in the kernel of the operator $\Delta_{n,v}$ on $k$-forms in the Sobolev setting are smooth, and thus this kernel can be identified with the space of normal $v$-harmonic fields $\mathcal{H}^k_{n,v}(M)$. In addition, since $M$ is compact, the operator $\Delta_{n,v}$ is Fredholm of index $0$ with finite-dimensional kernel. We immediately have the following theorem.
\begin{thm}
    The space $\mathcal{H}^k_{n,v}(M)$ is finite-dimensional.
\end{thm}
Correspondingly, we have the identification between the kernel of $\Delta_{t,v}$ on $k$-forms in the Sobolev setting and the space of tangential $v$-harmonic fields $\mathcal{H}^k_{t,v}(M)$, which is also finite-dimensional.

\begin{prop}\label{prop:uniquePotential}
    For each $\eta\in(\mathcal{H}^k_{n,v}(M))^{\perp}$, there is a unique solution $\omega\in H^2\bar\Omega^k_{n,v}(M)\cap(\mathcal{H}^k_{n,v}(M))^\perp$ to the boundary value problem~\eqref{eq:bvp}.
\end{prop}
\begin{proof}
    By the Fredholmness and self-adjointness of $\Delta_{n,v}$, its range is closed in $L^2\Omega^k(M)$, and agrees with the orthogonal complement of its kernel, i.e.,  $(\mathcal{H}^k_{n,v}(M))^{\perp} = \Delta_{n,v}(H^2\bar\Omega^k_{n,v}(M))$. This establishes the existence of a potential $\omega$ for the boundary value problem~\eqref{eq:bvp} when $\eta\in(\mathcal{H}^k_{n,v}(M))^{\perp}$. Now let $\omega_1$ and $\omega_2$ be two solutions of \eqref{eq:bvp} in $H^2\bar\Omega^k_{n,v}(M)\cap(\mathcal{H}^k_{n,v}(M))^\perp$. Then $$\omega_1-\omega_2\in\ker\Delta_{n,v} = \mathcal{H}^k_{n,v}(M).$$ Thus $\omega_1-\omega_2 = 0$, which proves the uniqueness.
\end{proof}







\subsection{$v$-induced Hodge decomposition on manifolds with boundary}




For manifolds with boundary, the Hodge-Morrey and Friedrichs decompositions provide the basic and refined splittings of the space of differential forms with respect to the standard Hodge $L^2$-inner product~\cite{schwarz2006hodge}. In this subsection, we establish the corresponding Hodge decompositions in our $v$-induced framework. The proofs follow from the ellipticity of the associated boundary value problems and are analogous to those in the standard case~\cite{schwarz2006hodge}, so we omit them here. We formulate the decomposition results at the $L^2$-level. However, they remain valid for smooth forms by the regularity of the elliptic operator.

We denote by
\begin{align}
    \mathcal{E}^k(M) &= \{d\alpha\,\vert\,\alpha\in H^1\Omega_n^{k-1}(M)\} \subset L^2\Omega^k(M)\\
    \mathcal{C}_v^k(M) &= \{\delta_v\alpha\,\vert\,\alpha\in H^1\Omega_{t,v}^{k+1}(M)\} \subset L^2\Omega^k(M)
\end{align}
the subspaces of exact normal and coexact $v$-tangential forms, respectively.
The $v$-induced Hodge-Morrey decomposition is then given below.
\begin{thm}[$v$-Hodge-Morrey decomposition]\label{thm.HodgeMorrey}
One has the following orthogonal decomposition
\begin{align}
    L^2\Omega^k(M) = \mathcal{E}^k(M)\oplus\mathcal{C}_v^k(M)\oplus L^2\mathcal{H}^k_{v}(M)
\end{align}
with respect to the $v$-induced inner product~\eqref{eq:HodgeL2.v.induced}.
\end{thm}
To establish the analogue of the Friedrichs decomposition, we consider the space of exact $v$-harmonic fields and the space of coexact $v$-harmonic fields, respectively, given by
\begin{align}
    \mathcal{H}^k_{v, \on{ex}}(M) &= \left\{h\in \mathcal{H}^k_v(M)\,\vert\, h = d\epsilon \right\},\\
    \mathcal{H}^k_{v, \on{co}}(M) &= \left\{h\in \mathcal{H}^k_v(M)\,\vert\, h = \delta_v\gamma \right\}
\end{align}
One then has the following $v$-Friedrichs decomposition.
\begin{thm}[$v$-Friedrichs decomposition]\label{thm.Friedrichs}
The space of $v$-harmonic fields can be further orthogonally split as follows
\begin{align}
    \mathcal{H}^k_v(M) = \mathcal{H}^k_{n,v}(M)\oplus \mathcal{H}^k_{v, \on{co}}(M) = \mathcal{H}^k_{t,v}(M)\oplus \mathcal{H}^k_{v, \on{ex}}(M)
\end{align}
with respect to the $v$-induced inner product~\eqref{eq:HodgeL2.v.induced}. 
\end{thm}

The following $4$-component Hodge decompositions are immediate consequences of Theorem~\ref{thm.HodgeMorrey} and~\ref{thm.Friedrichs}.
\begin{cor}[$v$-Hodge-Morrey-Friedrichs decomposition]\label{cor:4componentDecompostion}
    The space $L^2\Omega^k(M)$ can be $L^2$-orthogonal decomposed into
    \begin{align}
    L^2\Omega^k(M) &= \mathcal{E}^k(M)\oplus\mathcal{C}_v^k(M)\oplus\mathcal{H}^k_{n,v}(M)\oplus \mathcal{H}^k_{v, \on{co}}(M)\\
    &= \mathcal{E}^k(M)\oplus\mathcal{C}_v^k(M)\oplus\mathcal{H}^k_{t,v}(M)\oplus \mathcal{H}^k_{v, \on{ex}}(M)
    \end{align}
    with respect to the $v$-induced inner product~\eqref{eq:HodgeL2.v.induced}.
\end{cor}

\begin{rem}
While the decompositions above are unique, the corresponding potentials are not. for instance, one has a unique $v$-Hodge-Morry decomposition
\begin{align}
    \omega = d\alpha_n + \delta_v\beta_{t,v} + h,
\end{align}
where $\alpha_n\in H^1\Omega^{k-1}_n(M)$, $\beta_{t,v}\in\Omega^{k+1}_{t,v}(M)$ and $h\in\mathcal{H}^k_v(M)$. Any form $\alpha_n + \eta$ (resp. $\beta_{t,v} + \gamma$) serves as a potential for the same component if $\eta\in\ker d\cap H^1\Omega^{k-1}_n(M)$ (resp. $\gamma\in\ker\delta_v\cap H^1\Omega^{k+1}_{t,v}(M)$). This issue can be addressed by choosing $\alpha_n$ (resp. $\beta_{t,v}$) that has the least $v$-induced $L^2$-norm among all potentials. This leads to the gauge conditions $\alpha_n\in\img\delta_v$ (resp. $\beta_{t,v}\in\img d$), which is also equivalent to requiring that $\delta_v\alpha_n = 0$ and $\alpha_n\in(H^k_{n,v})^\perp$ (resp. $d\beta_{t,v} = 0$ and $\beta_{t,v}\in(H^k_{t,v})^\perp$). The potential $\alpha_n$ (resp. $\beta_{t,v}$) can then be solved uniquely by considering the corresponding boundary value problem $\Delta_v\alpha_{n,v} = \delta_v\omega$ (resp. $\Delta_{t,v}\beta_{t,v} = d\omega$) with the gauge conditions.
\end{rem}

\begin{rem}
    One may also expect a refined $5$-component decomposition given by
    \begin{align}
        \Omega^k(M) = d\Omega^{k-1}_{n,v}\oplus\delta_v\Omega^{k+1}_{t,v}\oplus\left(\mathcal{H}^k_{n,v}+\mathcal{H}^k_{t, v}\right)\oplus\left(d\Omega^{k-1}\cap\delta_v\Omega^{k+1}\right)
    \end{align}
     as in the standard case~\cite{deturck2004poincare}. To establish such decomposition, a detailed study of the interior and boundary subspaces of the space of normal and tangential fields is needed in our $v$-induced setting. We will not discuss it further. In the standard case for compact domains in a Euclidean space, the corresponding $5$-component decomposition becomes orthogonal \cite{shonkwiler2009poincare}. It is natural to expect an analogous result in our framework.
\end{rem}


\subsection{Relative and absolute de Rham $v$-cohomology}

    
We now return to the smooth setting to discuss the de Rham cohomology groups. 
The space of $v$-harmonic fields $\mathcal{H}^k_v(M)$ is infinite dimensional, and thus does not directly provide any topological information about the underlying manifold. To recover the Hodge isomorphism, we therefore restrict $\mathcal{H}^k_v(M)$ to its finite dimensional subspaces $\mathcal{H}^k_{n,v}(M)$ and $\mathcal{H}^k_{t,v}(M)$ under normal and $v$-tangential boundary conditions. 

Note that $d$ preserves the normal boundary conditions and $\delta_v$ preserves the $v$-tangential boundary conditions. The corresponding de Rham subcomplex and dual subcomplex are therefore well defined by restricting $d$ to the space of normal forms and $\delta_v$ to the space of $v$-tangential forms, respectively, given by
\begin{align}\label{eq:deRhamsubcomplexes}
\cdots \longrightarrow \Omega_n^{k-1}(M)
\xrightarrow{\,d\,} &\Omega_n^k(M)
\xrightarrow{\,d\,} \Omega_n^{k+1}(M)
\longrightarrow \cdots, \\
\cdots \longleftarrow \Omega_{t,v}^{k-1}(M)
\xleftarrow{\,\delta_v\,} &\Omega_{t,v}^k(M)
\xleftarrow{\,\delta_v\,} \Omega_{t,v}^{k+1}(M)
\longleftarrow \cdots.
\end{align}
These then lead to the relative and induced absolute de Rham $v$-cohomology groups
\begin{align}\label{eq:cohomology}
H^k_{dR}(M, \partial M) &= \ker d^k / \img d^{k-1}\\
H^k_{dR, v}(M) &= \ker \delta_v^k / \img \delta_v^{k+1}
\end{align}
with domains restricted to the space of normal forms and the space of $v$-tangential forms. Although the vector field $v$ appears in the definition of the $v$-induced absolute de Rham cohomology $H^k_{dR, v}(M)$, we show that this group is in fact independent of the choice of $v$ and isomorphic to the usual absolute de Rham cohomology.

The relative and absolute de Rham cohomology groups~\eqref{eq:cohomology} can also be formulated by considering sequences of the entire space of forms $\Omega^k(M)$ linked by $\delta_v$ and $d$, respectively. However, these definitions are equivalent following the $v$-Hodge-Morrey-Friedrichs decompositions as in the standard de Rham-Hodge theory~\cite{schwarz2006hodge}. For our purposes, we adopt~\eqref{eq:cohomology} to define the de Rham cohomology.

In the following, we give the Hodge isomorphism theorem for manifolds with boundary in our $v$-induced setting.

\begin{thm}[$v$-Hodge isomorphism]\label{thm:hodgeisomorphism}
    Each relative de Rham cohomology class in $H^k_{dR}(M, \partial M)$ contains a unique normal $ v$-harmonic field, while each absolute $ v$-cohomology class in $H^k_{dR, v}(M)$ admits a unique tangential $v$-harmonic field. Therefore, one has the natural isomorphisms
    \begin{align}
        H^k_{dR}(M, \partial M)\cong\mathcal{H}^k_{n,v}(M)\quad\text{and}\quad H^k_{dR, v}(M)\cong\mathcal{H}^k_{t,v}(M).
    \end{align}
\end{thm}
\begin{proof}
    By the $v$-Hodge-Morry decomposition~\ref{thm.HodgeMorrey}, each element $\omega\in\Omega^k(M)$ can be uniquely decomposed as follows
    \begin{align}
        \omega = d\alpha_n + \delta_v\beta_{t,v} + h,
    \end{align}
    where $\alpha_n\in\Omega^{k-1}_{n}(M)$, $\beta_{t,v}\in\Omega^{k+1}_{t,v}(M)$ and $h\in\mathcal{H}^k_v(M)$. If $\omega\in\ker d\,\big\vert_{\Omega^k_n(M)}$, then $\delta_v\beta_{t, v} = 0$ and thus $\omega = d\alpha_n + h$. In addition, $j^*\omega = 0$ implies $j^*h = 0$, i.e., $h\in\mathcal{H}^k_{n,v}(M)$ is a normal harmonic field. It follows directly that
    $H^k_{dR}(M, \partial M)\cong\mathcal{H}^k_{n,v}(M)$. The dual isomorphism $H^k_{dR, v}(M)\cong\mathcal{H}^k_{t,v}(M)$ can be shown in the same way. 
\end{proof}

We further have the following duality.
\begin{cor}\label{cor.poincaredual}
The $v$-induced Hodge star $\star_v$ induces an isomorphism
\begin{align}
    H^{k}_{dR, v}(M)\cong H^{m-k}_{dR}(M, \partial M).
\end{align}
\end{cor}
\begin{proof}
    Recall that the $v$-induced Hodge star $\star_v$ provides an isomorphism from the space of $v$-tangential harmonic forms $\mathcal{H}^k_{t, v}(M)$ to the space of $v$-normal harmonic forms $\mathcal{H}^{m-k}_{n, v}(M)$. The statement follows directly from Theorem~\ref{thm:hodgeisomorphism}.
\end{proof}

Note that the relative de Rham cohomology group $H^{k}_{dR}(M, \partial M)$ is defined independently of the vector field $v\in\mathcal{X}(M)$. By Corollary~\ref{cor.poincaredual}, together with its analogue in the standard case, the de Rham absolute $v$-cohomology group $H^{k}_{dR, v}(M)$ is independent of the choice of $v$, and is isomorphic to the usual de Rham absolute cohomology group. Furthermore, by the de Rham theorem, the dimensions of the space of tangential $v$-harmonic fields $\mathcal{H}^k_{t,v}(M) = \ker\Delta^k_{t,v}$ and the space of normal $v$-harmonic fields $\mathcal{H}^k_{n,v}(M) = \ker\Delta^k_{n,v}$ are given by the $k$-th Betti number $\beta_k$ and the $(m\!-\!k)$-th Betti number $\beta_{m-k}$, respectively. 






    

    

    


\section{Some remarks on the $v$-induced framework}

    

In this section, we discuss several features of the $v$-induced framework that may be of interest. We begin by comparing the $0$-th $v$-induced Hodge Laplacian with the corresponding $0$-th Hodge Laplacian (i.e., the Laplace-Beltrami operator) in the standard Hodge theory, and derive several simplified formulations under additional assumptions on $v$. Similar formulas can be obtained for higher-order forms, but they become quite involved so we do not pursue them here. We also present several explicit examples on $S^1$, $S^2$ and $[0,1]^2$ of $\Delta_v$ on functions, illustrating the spectral behavior of the induced Laplacian on functions. Next, we establish a condition under which a harmonic field $v$ remains $v$-harmonic in its own induced framework. Finally, we show that the $v$-induced framework is invariant under isometries.


\subsection{$0$-th order $v$-induced Hodge Laplacian}

The $0$-th order $v$-induced Hodge Laplacian and the standard Hodge Laplacian are given by $\Delta_v = \delta_vd$ and $\Delta = \delta d$ acting on $0$-forms (i.e., functions), respectively. Let $\varphi\in C^\infty(M)$. Notice that $\delta(\phi\omega) = \phi\delta\omega - i_{\nabla \phi}\omega$ and $T_v(\omega) = \omega + \omega(v)v^\flat$ for $\phi\in C^{\infty}(M)$ and $\omega\in\Omega^1(M)$. We have
\begin{align}
    \Delta_v\varphi = \delta_vd\varphi = \delta T_vd\varphi = \delta(d\varphi + v(\varphi)v^\flat) &= \Delta \varphi + v(\varphi)\delta v^\flat - i_{\nabla(v(\varphi))}v^\flat\\ \label{eq:0thLaplacian1}
    &= \Delta \varphi + v(\varphi)\delta v^\flat - v(v(\varphi)).
\end{align}
This formula can be further expanded in terms of the Hessian. Recall that the Hessian of a function $\phi$ is
\begin{align}\label{eq:hessian}
    \nabla^2\phi(X, Y) =\langle\nabla_X\nabla\phi, Y\rangle_g = X(Y(\phi)) - (\nabla_XY)\phi,\quad X, Y\in\mathcal{X}(M)
\end{align}
where $\nabla$ denotes the Levi-Civita connection associated with the metric $g$ on $M$. Substituting the above formula into~\eqref{eq:0thLaplacian1}, we obtain
\begin{align}\label{eq:0thLaplacian2}
    \Delta_v\varphi &= \Delta \varphi + v(\varphi)\delta v^\flat - (\nabla_vv)\varphi  - \nabla^2\varphi(v,v) 
\end{align}
One can see that the difference between $\Delta_v\varphi$ and $\Delta\varphi$ contains the second-order term $\nabla^2\varphi(v,v)$, while the remaining terms are of lower order.

The formula \eqref{eq:0thLaplacian1} can be simplified under additional assumptions on the vector field $v$, for instance, when $v$ is a gradient, coclosed, harmonic or parallel in its own direction. We do not discuss all cases separately here. Instead, we consider only the following example.

In the special case that $v = \nabla f$ is a gradient, so $v^\flat = df$. Using again the Hessian formula~\eqref{eq:hessian}, we have
\begin{align}
    \Delta_v\varphi &= \Delta \varphi + v(\varphi)\Delta f - \nabla^2\varphi(\nabla f, \nabla f) - (\nabla_{\nabla f}\nabla f)\varphi\\
    &= \Delta \varphi + \langle\nabla f, \nabla\varphi\rangle_g\Delta f - \nabla^2\varphi(\nabla f, \nabla f) - \langle\nabla_{\nabla f}\nabla f, \nabla\varphi\rangle_g\\
    &= \Delta \varphi + \langle\nabla f, \nabla\varphi\rangle_g\Delta f - \nabla^2\varphi(\nabla f, \nabla f) - \nabla^2f(\nabla f, \nabla\varphi)
\end{align}
If $v^\flat$ is further coclosed, which is equivalent to $f$ being harmonic, i.e., $\Delta f = \delta df = \delta v^\flat = 0$, the formula simiplies to
\begin{align}
    \Delta_v\varphi  = \Delta \varphi - \nabla^2\varphi(\nabla f, \nabla f) - \nabla^2f(\nabla f, \nabla\varphi)
\end{align}
In particular, if $\norm{v}_g = \norm{\nabla f}_g$ is constant, then $\nabla^2f(\nabla f, X) = \frac12X(\norm{\nabla f}^2_g) =0$ for any $X\in\mathcal{X}(M)$, and thus
\begin{align}
    \Delta_v\varphi  = \Delta \varphi - \nabla^2\varphi(\nabla f, \nabla f).
\end{align}



A local coordinate expression of the $0$-th $v$-induced Hodge Laplacian is given in Section~\ref{sec:innerProd_and_operators}, which shows that it falls into the class of anisotropic Laplace-Beltrami type operators~\cite{clarenz2000anisotropic, bajaj2003anisotropic}. In the formula, $T_v^{ij} = g^{ij} + v^iv^j$ is exactly the metric~\eqref{eq:pointwise_metric_g.vinduced} induced by $T_v$ on $1$-forms in local coordinates, obtained by perturbing the metric $g$ in the direction of $v$. As a result, the induced anisotropy enhances diffusion along the direction of $v$.

We now present several explicit examples on $S^1, S^2$, and $[0,1]^2$, where the eigenmodes and eigenvalues of $\Delta_v$ can be computed explicitly, illustrating the spectral effects of the vector field $v$.

\paragraph{Example 1}
Let $M=S^1$ with the metric $g = d\theta^2$ and let $v = c\partial_{\theta}$ be a constant vector field with $\theta\in[0,2\pi)$. Then $T_v^{11} = 1 + c^2$, and the $v$-induced Laplacian on functions is
\begin{align}
    \Delta_v \varphi =-(1+ c^2)\varphi_{\theta\theta} = (1+ c^2)\Delta\varphi, \qquad \varphi\in C^{\infty}(M),
\end{align}
where $\Delta\varphi = -\varphi_{\theta\theta}$ denotes the standard Laplace-Beltrami operator. It follows immediately that the eigenmodes of $\Delta_v$ on functions agree with those of $\Delta$ given by $\varphi_k=e^{ik\theta}$, while the corresponding eigenvalues are multiplied by the factor $1+c^2$, i.e., $\lambda_{v,k}=(1+c^2)\lambda_k = (1+c^2)k^2$. Here $\lambda_k$ denotes the eigenvalue of $\Delta$ corresponding to $\varphi_k$.

\paragraph{Example 2}
For the second example, we consider $M = S^2$ with the standard metric $g = d\theta^2 + \sin^2\theta d\phi^2$, where $\theta\in[0,\pi]$ and $\phi\in[0,2\pi)$ are the polar and azimuthal coordinates, respectively. The standard Laplacian on functions is
\begin{align}
    \Delta\varphi = -\frac{1}{\sin\theta}\partial_\theta(\sin\theta \varphi_\theta) - \frac{1}{\sin^2\theta}\varphi_{\phi\phi}
\end{align}
with eigenmodes given by the spherical harmonic functions $Y^m_l$ of degree $l$ and order $m$ and eigenvalues $\lambda_l = l(l+1)$. Let $v = c\partial_\phi$. In this case, the spherical harmonic functions $Y^m_l$ remain the eigenmodes of the $v$-induced Laplacian $\Delta_v$. In fact, we have 
\begin{align}
(T_v^{ij}) = \begin{pmatrix}
        1 & 0\\[.3em]
        0 & \frac{1}{\sin^2\theta} + c^2
    \end{pmatrix}.
\end{align}
The $v$-induced Laplacian on functions is 
\begin{align}
    \Delta_v\varphi = -\frac{1}{\sin\theta}\partial_\theta(\sin\theta\varphi_\theta) - \frac{1}{\sin^2\theta}\varphi_{\phi\phi} - c^2\varphi_{\phi\phi} = \Delta\varphi - c^2\varphi_{\phi\phi}.
\end{align}
Note that $\partial_{\phi\phi}Y^m_l = -m^2Y^m_l$. It follows that
\begin{align}
\Delta_vY^m_l = (l(l+1)+c^2m^2)Y^m_l
\end{align}
Thus, $Y^m_l$ are the eigenmodes of $\Delta_v$ with eigenvalues $\lambda_{v,l} = l(l+1)+c^2m^2$.

\paragraph{Example 3}
Now let $M = [0,1]^2$ with the metric $g = dx^2 + dy^2$ and let $v = c_1\partial_x + c_2\partial_y$. Then we have
\begin{align}
    (T_v^{ij}) = \begin{pmatrix}
        1+ c_1^2 & c_1c_2\\[.3em]
        c_1c_2 & 1+ c_2^2
    \end{pmatrix}.
\end{align}
It follows that the $v$-induced Laplacian on functions is 
\begin{align}
    \Delta_v\varphi &= - (1+ c_1^2)\varphi_{xx} - 2c_1c_2\varphi_{xy} - (1+ c_2^2)\varphi_{yy}\\
    &= \Delta\varphi - \left(c_1^2\varphi_{xx} + 2c_1c_2\varphi_{xy} + c_2^2\varphi_{yy}\right),
\end{align}
where $\Delta\varphi = -\varphi_{xx} - \varphi_{yy}$ is the standard Laplace-Beltrami operator. In the case that $v = c_1\partial_x$ is tangent to the horizontal edges and normal to the vertical edges, $T_v$ preserves the normal and tangential boundary conditions. Under the normal boundary condition $\varphi\vert_{\partial M} = 0$, the eigenmodes of the $v$-induced Laplacian $\Delta_v\varphi = \Delta\varphi - c_1^2\varphi_{xx} = -(1+c_1^2)\varphi_{xx}-c_2\varphi_{yy}$ are $\varphi_{m,n} = \sin(m\pi x)\sin(n\pi y)$, which agree with those of $\Delta$. The corresponding eigenvalues are $\lambda_{v,m,n} = \pi^2((1+c_1^2)m^2 +n^2)$, while the eigenvalues of $\Delta$ are $\lambda_{m,n} = \pi^2(m^2 +n^2)$.

The $v$-induced pointwise inner product~\eqref{eq:pointwise_metric_g.vinduced} on $1$-forms naturally leads to a $v$-induced metric $g_v$ on $M$. One may then define the Levi-Civita connection associated with $g_v$, and also the corresponding Ricci tensor, to formulate the $v$-induced Weitzenb{\"o}ck or Bochner formulas. This is, however, outside the scope of this paper. We leave it for future work.

\subsection{Harmonicity to $v$-harmonicity}

The $v$-induced framework is built using the vector field $v$. It is natural to ask when a harmonic field $v$ in the standard theory remains $v$-harmonic in its own induced framework. The following proposition gives a simple condition for this.

\begin{prop}\label{prop:harmonc.vharmonic}
    If $v^\flat$ is harmonic in the standard sense, i.e., $v^{\flat}\in\mathcal{H}^1(M)$, then it is $v$-harmonic $v^{\flat}\in\mathcal{H}^1_{v}(M)$ if and only if the squared norm of $v$ does not change in the direction of $v$ itself, i.e., $v(\norm{v}_g^2)=0$. In particular, this holds when $v$ has constant pointwise norm.
\end{prop}
\begin{proof}
    Let $v^{\flat}\in\mathcal{H}^1(M)$. Then $d\omega = \delta\omega = 0$. We compute
    \begin{align}
        T_v v^\flat = v^\flat + v^\flat\wedge\iota_vv^\flat = v^\flat + \norm{v}_g^2v^\flat.
    \end{align}
    Applying the codifferential $\delta$ to both sides yields
    \begin{align}
        \delta_v v^\flat = \delta T_v v^\flat = \delta v^\flat + \delta(\norm{v}_g^2v^\flat) = \norm{v}_g^2\delta v^\flat - \iota_{\nabla\norm{v}_g^2}v^\flat = v(\norm{v}_g^2).
    \end{align}
    The statement follows directly.
\end{proof}


\subsection{Invariance under isometry}

One may also be interested in whether two different vector fields induce the same $v$-induced framework or produce the $v$-induced Hodge Laplacians with the same spectrum. It is clear that two vector fields that differ by a sign lead to the same $T_v$ by its definition, and thus the same $v$-induced framework. A natural question is whether this happens in other situations. Since the Hodge Laplacian is an intrinsic differential operator that is invariant under isometries, one expects an analogous property in the $v$-induced setting. This is indeed the case.

\begin{prop}\label{prop:spectrumunderisometry}
    Let $\Phi: M\to M$ be an isometry and let $v_1$ and $v_2$ be two vector fields such that $v_2 = \Phi_*v_1$. Then the corresponding $v$-induced Hodge Laplacians satisfy $\Phi^*\Delta_{v_2}\omega = \Delta_{v_1}\Phi^*\omega$ for $\omega\in\Omega^k(M)$. Therefore, they share exactly the same eigenvalues, and the corresponding eigenforms are identified by the pullback $\Phi^*$.
\end{prop}
\begin{proof}
    Since the differential $d$, the codifferential $\delta$ commute with $\Phi^*$, it suffices to show that $\Phi^*T_{v_2}\omega = T_{v_1}\Phi^*\omega$ for $\omega\in\Omega^k(M)$. Note that $\Phi^*v_2^\flat = v_1^\flat$ and $\Phi^*\iota_{\Phi_*X}\omega = \iota_X\Phi^*\omega$ for $X\in\mathcal{X}(M)$. We obtain
    \begin{align}
        \Phi^* T_{v_2}\omega = \Phi^*\omega + \Phi^*\left(v_2^\flat\wedge\iota_{v_2}\omega \right) &= \Phi^*\omega + \Phi^*\left(v_2^\flat\right)\wedge\Phi^*\left(\iota_{v_2}\omega \right)\\
        &= \Phi^*\omega + v_1^\flat\wedge \iota_{v_1}\Phi^*\omega = T_{v_1}\Phi^*\omega,
    \end{align}
    which proves the statement.
\end{proof}
The statement of Proposition~\ref{prop:spectrumunderisometry} remains valid for the case of two different manifolds that differ by an isometry. This suggests that the spectrum of the $v$-induced Hodge Laplacian can serve as a shape descriptor for a vector field on a manifold. In particular, it is isometry-invariant, which is essential for comparing vector fields across manifolds.

\begin{appendices}

\section{Appendix}

In this appendix, we show that the operator $T_v$ preserves the normal and tangential boundary conditions only if $v$ is normal or tangential to the boundary. Note that $T_v$, however, is not linear in $v$, the result thus cannot be extended to an arbitrary vector field on $M$.

\begin{lem}\label{lem:interiorstar}
Let $\zeta\in\Omega^{k-1}(M)$. Then
    \begin{align} 
        \star(v^\flat\wedge\zeta) = (-1)^{k-1} \iota_v(\star\zeta)
    \end{align}
\end{lem}
\begin{proof}
    Let $\eta\in\Omega^k(M)$. Then $\eta\wedge\star\zeta = 0$, and thus
    \begin{align}
        0 = \iota_v(\eta\wedge\star\zeta) &= \iota_v(\eta)\wedge\star\zeta + (-1)^{k}\eta\wedge \iota_v(\star\zeta)
    \end{align}
    In addition, by Lemma~\ref{lem.wedge.interiorprod.adjoint}, we obtain
    \begin{align}
        \iota_v(\eta)\wedge\star\zeta = \langle\zeta, \iota_v(\eta)\rangle_g d\mu_g = \langle v^{\flat}\wedge\zeta, \eta\rangle_g d\mu_g = \eta\wedge\star(v^{\flat}\wedge\zeta).
    \end{align}
    Therefore,
    \begin{align}
        \eta\wedge\star(v^{\flat}\wedge\zeta) = (-1)^{k-1}\eta\wedge \iota_v(\star\zeta),
    \end{align}
    and the formula is proved.
\end{proof}

\begin{lem}\label{lem:starinterior}
Let $\omega\in\Omega^k(M)$. Then
    $$\star\iota_v\omega = (-1)^{k-1}v^\flat\wedge\star\omega$$
\end{lem}
\begin{proof}
    Note that $\star\omega\in\Omega^{n-k}(M)$. By Lemma~\ref{lem:interiorstar}, we obtain
    \begin{align}
        \star(v^\flat\wedge\star\omega) = (-1)^{n-k} \iota_v(\star\star\omega).
    \end{align}
    Applying the Hodge star $\star$ on both sides, the formula follows since $\star\star =(-1)^{k(n-k)}$ on $k$-forms.
\end{proof}

\begin{lem}\label{lem:starTv}
 One has the following formula
\begin{align}
        \star T_v = (1+\norm{v}_g^2)T_v^{-1}\star.
    \end{align}
\end{lem}
\begin{proof}
    Let $\omega\in\Omega^k(M)$. Using Lemma~\ref{lem:interiorstar}, we compute
    \begin{align}
        \star T_v \omega = \star\omega + \star\left(v^\flat\wedge\iota_v\omega \right) &= \star\omega + (-1)^{k-1}\iota_v\star\iota_v\omega.
    \end{align}
    It then follows from Lemma~\ref{lem:starinterior} that
    \begin{align}
        \iota_v\star\iota_v\omega = (-1)^{k-1}\iota_v\left(v^\flat\wedge\star\omega\right) = (-1)^{k-1}\left(\norm{v}_g^2\star\omega - v^\flat\wedge(\iota_v\star\omega)\right).
    \end{align}
    Therefore,
    \begin{align}
        \star T_v \omega = (1+ \norm{v}_g^2)\left(\id - \frac{1}{1+ \norm{v}_g^2}v^\flat\wedge\iota_v\right)\star\omega = (1+ \norm{v}_g^2) T_v^{-1}\star\omega,
    \end{align}
    which proves the formula.
\end{proof}

\begin{lem}\label{lem:starstar}
let $\omega\in\Omega^k(M)$. Then
    \begin{align}
        \star_v\star_v\omega = (-1)^{k(n-k)}\left(1+\norm{v}_g^2\right)\omega
    \end{align}
\end{lem}
\begin{proof}
    The formula follows directly from the definition of $\star_v = \star T_v$ and Lemma~\ref{lem:starTv}.
\end{proof}

\begin{prop}\label{prop:Tv.bcs}
    The operator $T_v$ preserves the normal and tangential boundary conditions if $v$ is normal or tangent to the boundary $\partial M$.
\end{prop}
\begin{proof}
First, we let $\omega\in\Omega_n^k(M)$ be a normal form, i.e., $j^*(\omega) = 0$.
By computation, we have
\begin{align}
    j^*(T_v\omega) &= j^*(\omega) + j^*\left(v^\flat\wedge \iota_v\omega\right) = (j^*v^\flat)\wedge \left(j^*\iota_v\omega\right)\\
\end{align}
If $v$ is normal to $\partial M$, then $j^*v^\flat = 0$. If $v$ is tangent to $\partial M$, then
\begin{align}
    j^*(\iota_v\omega)(X_1,\ldots,X_{k-1}) &= (\iota_v\omega)(X_1,\ldots,X_{k-1})\\
    &=\omega(v, X_1,\ldots,X_{k-1}) = 0
\end{align}
holds for all $X_i\in\mathcal{X}(\partial M), i = 1,\ldots, k-1$ since $\omega$ is normal.
Both cases lead to $j^*(T_v\omega) = 0$, which means that $T_v\omega$ is normal.

Now we consider the case that $\omega\in\Omega_t^k(M)$ is tangential, i.e., $j^*(\star\omega) = 0$. By Lemma~\ref{lem:starTv}, we have
\begin{align}
    j^*(\star T_v\omega) &= (1+ \norm{v}_g^2)j^*\star\omega - j^*\left(v^\flat\wedge(\iota_v\star\omega)\right) = - j^*(v^\flat)\wedge j^*\left(\iota_v\star\omega\right).
\end{align}
Note that $\star\omega$ is a normal form. By the same reasoning for normal forms above, we immediately obtain $j^*(\star T_v\omega) = 0$ when $v$ is normal or tangential, and thus $\star T_v\omega$ is tangential. The statement is then proved. 
\end{proof}



\end{appendices}

\bibliography{refs}

\end{document}